\documentclass[12pt,a4paper,psamsfonts]{amsart}

\usepackage{fancyhdr}
\usepackage{appendix}
\usepackage{amssymb,amscd,amsxtra,calc}
\usepackage{mathrsfs}
\usepackage{cmmib57}
\usepackage{multirow}
\usepackage[all]{xy}
\usepackage{longtable}
\usepackage[colorlinks,linkcolor=blue,anchorcolor=blue,citecolor=green]{hyperref}
\usepackage{tikz}

\makeatletter
\@namedef{subjclassname@2020}{\textup{2020} Mathematics Subject Classification}
\makeatother

\theoremstyle{plain}
    \newtheorem{thm}{Theorem}[section]

     \newtheorem{conjecture}[thm]{Conjecture}
    \newtheorem{corollary}[thm]{Corollary}
    
    \newtheorem{lemma}[thm]{Lemma}
    \newtheorem{proposition}[thm]{Proposition}
    \newtheorem{question}[thm]{Question}
    \newtheorem{theorem}[thm]{Theorem}

\theoremstyle{definition}
\newtheorem{blank}[thm]{}

    \newtheorem{notation}[thm]{Notation}
    \newtheorem*{notation*}{Notation and Terminology}
      
    \newtheorem{remark}[thm]{Remark}
    \newtheorem*{ack}{Acknowledgments}
\theoremstyle{remark}

\newcommand{\Q}{\mathbb{Q}}

\newcommand{\R}{\mathbb{R}}

\newcommand{\id}{\operatorname{id}}

\newcommand{\NE}{\overline{\operatorname{NE}}}
\newcommand{\Nef}{\operatorname{Nef}}

\newcommand{\NS}{\operatorname{NS}}

\newcommand{\PE}{\operatorname{PE}}

\newcommand{\Supp}{\operatorname{Supp}}

\newcommand{\N}{\operatorname{N}}

\newcommand{\Sym}{\operatorname{Sym}}

\newcommand{\Pic}{\operatorname{Pic}}

\makeatletter

\newcommand{\Rmnum}[1]{\expandafter\@slowromancap\romannumeral #1@}
\makeatother

\begin{document}

\title[Threefolds admitting polarized endomorphisms]
{Log Calabi-Yau structure of projective threefolds admitting polarized endomorphisms}

\author{Sheng Meng}

\address{
    \textsc{School of Mathematical Sciences, Shanghai Key Laboratory of PMMP}\endgraf
    \textsc{East China Normal University, 500 Dongchuan Road, Shanghai 200241, People's Republic of China;}\endgraf
	\textsc{Korea Institute For Advanced Study,
		Seoul 02455, Republic of Korea}
}
\email{smeng@math.ecnu.edu.cn; ms@u.nus.edu}

\subjclass[2020]{
14E30,   
32H50, 
20K30 
}

\keywords{polarized endomorphism, Calabi-Yau type, equivariant minimal model program, ramification divisor}

\maketitle

\begin{center}
{\em Dedicated to Professor De-Qi Zhang on his sixtieth birthday}
\end{center}

\begin{abstract}
Let $X$ be a normal projective variety admitting a polarized endomorphism $f$, i.e., $f^*H\sim qH$ for some ample divisor $H$ and integer $q>1$.
It was conjectured by Broustet and Gongyo that $X$ is of Calabi-Yau type, i.e., $(X,\Delta)$ is lc for some effective $\Q$-divisor such that $K_X+\Delta\sim_{\Q} 0$.
In this paper, we establish a general guideline based on the equivariant minimal model program and the canonical bundle formula.
In this way, we prove the conjecture when $X$ is a smooth projective threefold.
\end{abstract}

\tableofcontents

\section{Introduction}

We work over an algebraically closed field of characteristic 0.
Let $f:X\to X$ be a surjective endomorphism of a normal projective variety $X$.
In the curve case, it is well known by the Hurwitz formula that $X$ is either a rational curve or an elliptic curve when $f$ is non-isomorphic.
This is equivalent to saying that the anticanonical divisor $-K_X$ is effective.
In higher dimensional case, to easily eliminate the distraction of automorphism, one focuses on the {\it polarized} endomorphism $f$, i.e., $f^*H\sim qH$ for some ample divisor $H$ and integer $q>1$.
Then by making use of the ramification divisor formula, Zhang and the author showed that $-K_X$ is effective when $X$ is $\Q$-Gorenstein. 
However, the effectivity of the anticanonical divisor, though important, can say very few on the detailed characterization of higher dimensional varieties.

A delicate operation is running the $f$-equivariant (after iteration) minimal model program. The smooth surface case is settled by Nakayama and the general higher dimensional situation is settled by Zhang and the author (cf.~\cite{Nak02},\cite{MZ18},\cite{MZ22},\cite{MZ20},\cite{CMZ20}).
In this way, Broustet and Gongyo \cite{BG17} proposed the following conjecture and proved the surface case.
Recall that a normal projective variety $X$ is of {\it Calabi-Yau type} if $(X,\Delta)$ is an lc pair for some effective Weil $\Q$-divisor $\Delta$ such that $K_X+\Delta\sim_{\Q} 0$.
By the abundance (cf.~\cite[Theorem 1.2]{Gon13}), the latter condition is equivalent to $K_X+\Delta\equiv 0$.
We also call the pair $(X, \Delta)$ {\it log Calabi-Yau}.

\begin{conjecture}\label{main-conj-cy}
Let $X$ be a normal projective variety admitting a polarized endomorphism.
Then $X$ is of Calabi-Yau type.
\end{conjecture}

In higher dimensional cases, Conjecture \ref{main-conj-cy} has been recently verified for rationally connected smooth projective varieties by Yoshikawa; see \cite{Yos20} or Theorem \ref{thm-yos}.
The main purpose of this paper is to give a partial guideline on Conjecture \ref{main-conj-cy} and to provide a full solution for the case of smooth projective threefolds.
The following is our main result.

\begin{theorem}\label{main-thm-lcy}
Let $X$ be a smooth projective threefold admitting a polarized endomorphism $f$.
Then $X$ is of Calabi-Yau type, i.e., $(X,\Delta)$ is an lc pair for some effective $\Q$-divisor $\Delta$ with $K_X+\Delta\sim_{\Q} 0$.
\end{theorem}

We briefly explain the strategy and difficulty in our proof.
We first run the $f$-equivariant minimal model program which ends up with a $Q$-abelian variety $Y$, i.e., a quasi-\'etale quotient of an abelian variety.
Since the smooth rationally connected case has been verified by Yoshikawa, we may assume $\dim(Y)>0$.
Then we observe the fibration $\pi:X\to Y$ and its $f$-periodic general fibre.
By applying the canonical bundle formula and ramification divisor formula, the natural idea is to reduce the problem to the following Conjecture \ref{main-conj-lc} proposed by Yoshinori Gongyo.
\begin{conjecture}[Gongyo]\label{main-conj-lc}
Let $f:X\to X$ be a $q$-polarized endomorphism of a smooth projective variety $X$.
Then $(X,\frac{R_f}{q-1})$ is an lc pair after iteration.
\end{conjecture}

However, for the surface case, we cannot fully prove Conjecture \ref{main-conj-lc} for $\mathbb{P}^2$, which is the only left case.
So the main difficulty of proving Theorem \ref{main-thm-lcy} remains in the case when $\pi:X\to Y$ is a $\mathbb{P}^2$-bundle over an elliptic curve $Y$.
By applying the Iitaka fibration of the anticanonical divisor, we only need to focus on the case when $-K_X$ is big; see Section \ref{sec-big}.
This new condition allows us to reduce the problem to the only very concrete case when $X\cong\mathbb{P}_Y(\mathcal{F}_2\oplus \mathcal{L})$ where $\mathcal{F}_2$ is the unique indecomposable rank 2 vector bundle with non-trivial global sections and $\mathcal{L}$ is a line bundle of negative degree.

\begin{ack}
The author would like to thank professor Paolo Cascini for the warm hospitality and inspiring discussion during visiting Imperial Colledge London in 2018.
The author would like to thank Doctor Guolei Zhong and the referee for valuable suggestions to improve this paper.
The author is supported in part by Science and Technology Commission of Shanghai Municipality (No. 22DZ2229014) and by a Research Fellowship of KIAS (MG075501).
\end{ack}

\section{Preliminaries}

We use the following notation throughout this paper.
\begin{notation}\label{notation2.1}
Let $X$ be a projective variety.

\begin{itemize}
\item The symbols $\sim$ (resp.~$\sim_{\mathbb Q}$, $\equiv$)  denote
the \textit{linear equivalence} (resp.~\textit{$\mathbb Q$-linear equivalence}, \textit{numerical equivalence}) on $\Q$- (or $\R$-) Cartier divisors.
We also use $\equiv$ to denote the {\it numerical equivalence} of $1$-cycles on $X$.

\item Denote by $\textup{NS}(X) = \Pic(X)/\Pic^0(X)$  the {\it N\'eron-Severi group} of $X$.
Let
 $\N^1(X):=\NS(X)\otimes_\mathbb{Z}\mathbb{R}$ the space of $\mathbb{R}$-Cartier divisors modulo  numerical equivalence and $\rho(X) :=\dim_{\mathbb{R}}\N^1(X)$ the
{\it Picard number} of $X$.
Let $\N_1(X)$ be the dual space of $\N^1(X)$ consisting of 1-cycles.
Denote by $\Nef(X)$ the cone of {\it nef divisors} in $\N^1(X)$ and $\NE(X)$ the dual cone consisting of {\it pseudo-effective 1-cycles} in $\N_1(X)$.

\item Let $f:X\to X$ be a surjective endomorphism. A subset $Y\subseteq X$ is {\it $f^{-1}$-invariant} (resp.~{\it $f^{-1}$-periodic})  if $f^{-1}(Y)=Y$ (resp.~$f^{-s}(Y)=Y$ for some  $s>0$).

\item A surjective endomorphism $f:X\to X$ is \textit{$q$-polarized} if $f^*H\sim qH$ for some ample Cartier divisor $H$ and integer $q>1$; see \cite[Propositions 1.1]{MZ18} for the equivalent definitions.

\item A smooth projective variety $X$ is {\it rationally connected} if any two general points of $X$ can be connected by a chain of rational curves. 

\item A normal projective variety $X$ is of \textit{Fano type}, if there is an effective Weil $\mathbb{Q}$-divisor $\Delta$ on $X$ such that the pair $(X,\Delta)$ has at worst klt singularities and $-(K_X+\Delta)$ is ample and $\mathbb{Q}$-Cartier.
If $\Delta=0$, we say that $X$ is a  \textit{(klt) Fano variety}.

\item Let $Y$ be a projective variety and $\mathcal{E}$ a vector bundle of rank $n$.
Denote by $\pi:\mathbb{P}_Y(\mathcal{E})\to Y$ the projective bundle of hyperplanes in $\mathcal{E}$ (not lines in $\mathcal{E}$), so that $\pi_*\mathcal{O}_{\mathbb{P}_Y(\mathcal{E})}(1) = \mathcal{E}$.
\end{itemize}
\end{notation}

The following two lemmas are well-known and useful.
\begin{lemma}\label{lem-cy-bir}
Let $\pi:X\to Y$ be a birational morphism of two normal projective varieties.
Then $Y$ is of Calabi-Yau type if $X$ is of Calabi-Yau type.
\end{lemma}
\begin{proof}
Suppose the pair $(X, \Delta_X)$ is log Calabi-Yau.
Let $\Delta_Y:=\pi_*\Delta_X$.
Then $K_Y+\Delta_Y=\pi_*(K_X+\Delta_X)\sim_{\mathbb{Q}} 0$.
Note that $K_X+\Delta_X=\pi^*(K_Y+\Delta_Y)$.
By \cite[Lemma 3.38]{KM98}, $(Y,\Delta_Y)$ has singularities not worse than $(X,\Delta_X)$.
So $(Y,\Delta_Y)$ is lc.
\end{proof}

\begin{lemma}\label{lem-cy-qe}
Let $\pi:X\to Y$ be a quasi-\'etale finite surjective morphism of normal projective varieties.
Then $X$ is of Calabi-Yau type if and only if so is $Y$.
\end{lemma}
\begin{proof}
Suppose $(Y, \Delta_Y)$ is log Calabi-Yau.
Let $\Delta_X=\pi^*\Delta_Y$.
Since $\pi$ is quasi-\'etale, $K_X=\pi^*K_Y$ and hence $(X, \Delta_X)$ is log Calabi-Yau by \cite[Proposition 5.20]{KM98}.

Conversely, assume that $(X, \Delta_X)$ is log Calabi-Yau.
Note that the Galois closure of $\pi$ is still quasi-\'etale by \cite[Theorem 3.7]{GKP16}.
So we may assume that $\pi$ is the quotient map of $X$ by a finite group $G$.
Let $\Delta=\frac{1}{|G|}\sum\limits_{g\in G} g^*\Delta_X$ and $\Delta_Y=\frac{1}{\deg \pi}\pi_*\Delta$.
Then $\Delta=\pi^*\Delta_Y$.
Let $\phi:W\to X$ be a log resolution of $(X,\Delta)$ and $E$ a $\phi$-exceptional prime divisor.
Then we have the discrepency 
\begin{align*}
 a(E,X,\Delta)&=mult_E(K_W-(\phi^*K_X+\frac{1}{|G|}\sum\limits_{g\in G} g^*\Delta_X))\\
&=\frac{1}{|G|}\sum\limits_{g\in G}  mult_E(K_W-(\phi^*K_X+g^*\Delta_X))\\
&=\frac{1}{|G|}\sum\limits_{g\in G} a(E,X,g^*\Delta_X)\ge -1
\end{align*}
where $mult_E$ means the coefficient of the divisor on $E$ and the last inequality holds by noting that each pair  $(X,g^*\Delta_X)$ is log Calabi-Yau for any $g\in G$.
Furthermore, we have
$$K_X+\Delta=\frac{1}{|G|}\sum\limits_{g\in G} (K_X+g^*\Delta_X)=\frac{1}{|G|}\sum\limits_{g\in G} g^*(K_X+\Delta_X)\sim_{\Q} 0.$$
So  $(X,\Delta)$ is log Calabi-Yau.
Since $\pi$ is quasi-\'etale by the assumption, we have $K_X=\pi^*K_Y$ and hence
$$K_X+\Delta=\pi^*(K_Y+\Delta_Y).$$
Therefore, $(Y,\Delta_Y)$ is log Calabi-Yau by  \cite[Proposition 5.20]{KM98}.
\end{proof}

\begin{proposition}\label{prop-kappa0}
Let $f:X\to X$ be a polarized endomorphism of a projective variety.
Let $D$ be an effective $\Q$-Cartier divisor on $X$ with $\kappa(X, D)=0$.
Then $\Supp D$ is $f^{-1}$-periodic.
\end{proposition}
\begin{proof}
By \cite[Proposition 3.7]{MZ22}, we have 
$$f^*f_*D\sim_{\Q} f_*f^*D=(\deg f) D.$$
Since $\kappa(X, D)=0$, we have $f^*f_*D=(\deg f) D$.
In particular, $f^{-i}f^i(\Supp D)=\Supp D$ for all $i\ge 0$.
By \cite[Lemma 8.1]{Men20}, $\Supp D$ is $f^{-1}$-periodic.
\end{proof}

A $\Q$-divisor $D$ is said to be {\it $\Q$-movable} if for any prime divisor $\Gamma$, one has $\Gamma\not\subseteq \Supp D'$ for some effective $\Q$-divisor $D'\sim_{\Q} D$.

Let $f:X\to X$ be a polarized endomorphism of a normal projective variety.
Denote by 
$T_f$
the finite union of $f^{-1}$-periodic prime divisors (cf.~\cite[Corollary 3.8]{MZ20}).
Denote by $P_f:=-(K_X+T_f)$.

\begin{proposition}\label{prop--kkappa0}
Let $X$ be a $\Q$-Gorenstein normal projective variety admitting a polarized endomorphism $f$.
Assume further $\kappa(X, -K_X)=0$.
Then $\Supp R_f=T_f$, $-K_X\sim_{\Q} T_f$ and $(X, T_f)$ is lc.
\end{proposition}
\begin{proof}
This follows from \cite[Theorem 6.2]{MZ22} and \cite[Theorem 1.3]{Zha13}.
\end{proof}

We generalize \cite[Theorem 1.5]{MZ22}.
\begin{theorem}\label{thm-move}
Let $f:X\to X$ be a polarized endomorphism of a $\Q$-factorial normal projective variety.
Then $-(K_X+T_f)$ is $\Q$-movable.
\end{theorem}
\begin{proof}
After iteration, we may assume each component of $T_f$ is $f^{-1}$-invariant.
Consider the log ramification divisor formula
$$f^*(-(K_X+T_f))=-(K_X+T_f)+\Delta_f$$
where $\Delta_f=R_f-(q-1)T_f$ is effective and contains no common component of $T_f$.
Replacing $-K_X$ by $-(K_X+T_f)$ in the proof of \cite[Theorem 6.2]{MZ22},
we have that $-(K_X+T_f)$ is effective.
Let $P$ be a prime divisor such that $P\subseteq \Supp D$ for any effective $\Q$-divisor $D\sim_{\Q} -(K_X+R_f)$.
Then $\kappa(X,P)=0$.
By Proposition \ref{prop-kappa0}, $P\subseteq T_f$.
For any effective $\Q$-divisor $D'\sim_{\Q} f^*(-(K_X+T_f))$, we have $\frac{1}{\deg f}f_*D'\sim_{\Q} -(K_X+T_f)$ and hence $P\subseteq \Supp D'$.

Let $V$ be the subspace of $\Pic_{\mathbb{R}}(X)$ spanned by $(f^i)^*\Delta_f$ and $C$ the cone generated $\Q$-movable divisors in $V$.
By \cite[Proposition 3.7]{MZ22}, $V$ is $f^*|_{\Pic_{\mathbb{R}}(X)}$-invariant and finite dimensional.
Note that $f^*(C)= C$.
Let $P$ be a prime divisor such that $P\subseteq \Supp D$ for any effective $\Q$-divisor $D\sim_{\Q} \Delta_f$.
Then $\kappa(X,P)=0$.
By Proposition \ref{prop-kappa0}, $P\subseteq T_f$.
Note that $\Delta_f$ has no common component of $T_f$.
So such $P$ does not exist and hence $\Delta_f\in C$.
A similar argument shows that $f^*(C)= C$.
By \cite[Proposition 3.2]{Men20}, we have $-(K_X+T_f)\in C$.
\end{proof}

We recall the following result by Yoshikawa \cite[Proposition 6.2]{Yos20} that Fano type can be preserved by (polarized) equivariant birational map.
\begin{proposition}\label{prop-bir-fanotype}
Let $f:X\to X$ be a polarized endomorphism of a normal projective variety $X$.
Let $\pi:X\dashrightarrow Y$ be an $f$-equivariant birational map.
Then $X$ is of Fano type if and only if so is $Y$.
\end{proposition}
\begin{proof}
We may simply take $\Delta=\Gamma=0$ in \cite[Proposition 6.2]{Yos20}.
\end{proof}

We recall the nice result of Yoshikawa \cite[Corollary 1.4]{Yos20}.
\begin{theorem}\label{thm-yos}
Let $X$ be a rationally connected smooth projective variety admitting a polarized endomorphism.
Then $X$ is of Fano type.
\end{theorem}

\section{Singularities of ramification divisor}
In this section, we give some general results concerning Conjecture \ref{main-conj-lc}.
First, we consider the coefficients of the ramification divisor.
\begin{theorem}\label{thm-coefficient}
Let $f:X\to X$ be a $q$-polarized endomorphism of a projective variety.
Then after iteration, the coefficient $r_P$ of $R_f$ on each prime divisor $P$ has $r_P\le q-1$ and the equality holds if and only if $P$ is $f^{-1}$-periodic.
\end{theorem}
\begin{proof}
Let $b$ be the number of prime divisors contained in the branch locus of $f$.
Let $P$ be a prime divisor and denote by $m_P:=mult_P$ the coefficient of a divisor on $P$.
Note that $f^*P$ is a reduced divisor when $P$ is not contained in the branch locus of $f$ (in particular, $m_P f^*P=1$ in this case).
So we may set
$$c:=\max\{m_P f^*P\,|\, P \text{ a prime divisor}\}$$
which is a finite number.
Choose some $a>0$ such that $c^b<q^{a/2}$.
Choose some $s>0$ such that $(\frac{q^t}{q^t-1})^{s/t}>c^b$ for any $1\le t<a$.
Let $N=\max\{a,s\}$ and take $n>N$.
Let $P$ be a prime divisor which is not $f^{-1}$-periodic.
Let $Q$ be a prime divisor in $f^{-n}(P)$.
Let $r=m_P (f^n)^*P$ be the coefficient of $(f^n)^*P$ on $Q$.
We shall show that $r<q^n$.

If $P$ is not $f$-periodic, then  $f^{-i}(P)$ and  $f^{-j}(P)$ have no common irreducible component for any $i,j>0$ with $i\neq j$.
Then $$r=m_Q  (f^n)^*P=\prod_{i=0}^{n-1} m_{f^i(Q)}f^*(f^{i+1}(Q))\le c^b<q^{a/2}<q^n$$
where the first inequality is because there are at most $b$ terms of $f^i(Q)$ contributing $m_{f^i(Q)}f^*(f^{i+1}(Q))>1$.
In the following, we consider the case when $P$ is $f$-periodic with period $t\ge 1$.

Let $r_1$ be the coefficient of $(f^t)^*P$ on $P$.
Since all irreducible components of $\{f^{-i}(P)\}_{0<i\le t}$ are different with each other, we have $r_1<c^b<q^{a/2}$.
Moreover, we claim that $r_1<q^t$.
Suppose the contrary that we can write $(f^t)^*P=q^tP+E$ for some effective divisor $E$.
By the projection formula, 
$$(\deg f^t)P=(f^t)_*(f^t)^*P=(f^t)_*(q^tP+E)=q^t(\deg f^t|_P)P+(f^t)_*E.$$
Note that $\deg f^t=q^{t\cdot \dim(X)}$ and $\deg f^t|_P=q^{t\cdot (\dim(X)-1)}$.
Therefore, we have $E=0$ and hence $P$ is $f^{-1}$-periodic, a contradiction.
So the claim is proved.

Let $e$ be the minimal non-negative integer such that $f^e(Q)=P$.
Let $r_2$ be the coefficient of $(f^e)^*P$ on $Q$.
Then $r=r_2\cdot r_1^{(n-e)/t}$.
If $t\ge a$, then $r<c^b\cdot (q^{a/2})^{n/a}<q^n$.
If $t<a$, then $r<c^b\cdot (q^t-1)^{n/t}<q^n$.
\end{proof}

\begin{remark}
The proof of Theorem \ref{thm-coefficient} can be easily applied to consider general surjective endomorphisms, though the statement could be a bit wordy. We leave the readers the pleasure of the coefficient estimate.
\end{remark}

We observe the behaviour of ramification divisors between equivariant birational morphisms.

\begin{proposition}\label{prop-induction-birational}
Let $\pi:X\to Y$ be a birational morphism of $\Q$-factorial normal projective varieties.
Let $f:X\to X$ and $g:Y\to Y$ be two $q$-polarized endomorphisms such that $g\circ \pi=\pi\circ f$.
Then $$K_X+\frac{R_{f}}{q-1}=\pi^*(K_Y+\frac{R_g}{q-1}).$$
In particular, $(X, \frac{R_f}{q-1})$ is lc if $(Y, \frac{R_g}{q-1})$ is lc.
\end{proposition}
\begin{proof}
Since $K_X$ and $K_Y$ are $\Q$-Cartier,
we may write $$K_X=\pi^*K_Y+E$$
where $E$ has supports contained in the exceptional divisor of $\pi$. 
By the ramification divisor formula, we have
$$K_X=f^*K_X+R_f\text{ and } K_Y=g^*K_Y+R_g.$$
By the above three equations, we have
$$\pi^*K_Y+E=f^*(\pi^*K_Y+E)+R_f=\pi^*g^*K_Y+f^*E+R_f=\pi^*(K_Y-R_g)+f^*E+R_f.$$
Note that $f^*E=qE$.
So we have
$$\pi^*\frac{R_g}{q-1}-\frac{R_f}{q-1}=E.$$
Then we have 
$$K_X+\frac{R_{f}}{q-1}=\pi^*(K_Y+\frac{R_g}{q-1})$$
as desired.
\end{proof}

The canonical bundle formula plays a key role in our reduction.
\begin{proposition}\label{prop-cbf-rf}
Let $\pi:X\to Y$ be an algebraic fibration where $X$ is an $n$-dimensional smooth projective variety and $Y$ is normal projective with $K_Y\equiv 0$.
Let $f:X\to X$ and $g:Y\to Y$ be $q$-polarized endomorphisms such that $g\circ \pi=\pi\circ f$ and $f^*K_X\equiv qK_X$. 
Then $(X, \frac{R_f}{q-1})$ is log Calabi-Yau after iteration if Conjecture \ref{main-conj-lc} holds true for $f^s|_F:F\to F$ where $F$ is a $f$-periodic (of period $s$) general fibre of $\pi$.
\end{proposition}
\begin{proof}
By the ramification divisor formula, we have
$$K_X+\frac{R_{f}}{q-1}\equiv 0$$
which holds true after arbitrary iteration.

Let $F$ be a general $f$-invariant smooth fibre of $\pi$ after iteration (cf.~\cite[Theorem 5.1]{Fak03}).
Note that $f|_F$ is $q$-polarized and $R_{f|_F}=R_f|_F$.
By the assumption, $(F, \frac{R_{f|_F}}{q-1})$ is lc after iteration.
So $(X, \frac{R_f}{q-1})$ is lc over the generic point of $\pi$.
By the lc canonical bundle formula \cite[Theorem 4.1.1]{Fuj04},
we have 
$$0\equiv K_X+\frac{R_{f}}{q-1}=\pi^*(K_Y+B+M)$$
where $B$ is effective and $M$ is pseudo-effective.
Note that $K_Y\equiv 0$.
So $B=0$ and hence $(X, \frac{R_f}{q-1})$ is lc.
\end{proof}

\section{Surface case}

We first prove Conjecture \ref{main-conj-lc} for surfaces except $\mathbb{P}^2$.
\begin{theorem}\label{thm-surf}
Let $f:X\to X$ be a $q$-polarized endomorphism of a smooth projective surface with $\rho(X)>1$.
Then $(X, \frac{R_f}{q-1})$ is lc after iteration.
\end{theorem}
\begin{proof}
We apply \cite[Theorem 1.8]{MZ18}.
If $K_X$ is pseudo-effective, then $R_f=0$ and we are done.
In the following, we may assume $K_X$ is not pseudo-effective.
After iteration, we may run an $f$-equivariant minimal model program of $X$ which ends up with a Fano contraction $\pi:X'\to Y$ with $Y$ being a curve of genus $\le 1$.
Then $f^*|_{\N^1(X)}=q\id$ and hence $K_X+\frac{R_f}{q-1}\equiv 0$ by the ramification divisor formula.
By Proposition \ref{prop-induction-birational}, we may assume $X=X'$.
Note that $X$ is a ruled surface over $Y$.
We finish the proof with the following 2 cases.

\textbf{Case 1.} Suppose $Y$ is elliptic. 
Note that Conjecture \ref{main-conj-lc} holds for curves and $f^*|_{\N^1(X)}=q\id$.
So we are done by Proposition \ref{prop-cbf-rf}.

\textbf{Case 2.} Suppose $Y=\mathbb{P}^1$.
Then $X\cong F_d$ with $d\ge 0$.
If $d=0$, then $R_f$ is simply a sum of fibres and horizontal sections with coefficients $\le q-1$ and hence this case is simple.
We may assume now that $d>0$.
Then $\pi$ admits a unique negative section curve $C$. 
Note that $f^{-1}(C)=C$.
Then we have the decomposition by effective $\Q$-divisors
$$\frac{R_f}{q-1}=C+V+H$$ 
where $V$ has support in the fibres of $\pi$ and each component of $H$ dominates $Y$.
Clearly, $C$ is not a component of $H$ and  each component of $V$ has the coefficient $\le 1$.
Fix a point $y\in Y$.
Let $F:=\pi^{-1}(y)$.
For the restriction $$f|_F:F\cong \mathbb{P}^1\to f(F)\cong \mathbb{P}^1,$$ 
we have
$\frac{R_{f|_F}}{q-1}=\frac{R_f}{q-1}|_F=C|_F+H|_F$.
Note that $\frac{R_f}{q-1}\cdot F=2$, each component of $\frac{R_{f|_F}}{q-1}$ has coefficient $\le 1$, and $C|_F$ is a reduced point.
Consider the local intersection number at $x\in F$.
If $x= C\cap F$, then $(H\cdot F)_x=0$ and hence 
$$((C+H)\cdot F)_x=1.$$
If $x\neq C\cap F$, then 
$$((C+H)\cdot F)_x\le (C+H)\cdot F-(C\cdot F)_{C\cap F}=\frac{R_f}{q-1}\cdot F-1=1.$$
In particular, we have $(X,F+C+H)$ is lc near $F$ by \cite[Corollary 5.57]{KM98}.
Note that $\frac{R_{f|_F}}{q-1} \le F+C+H$ near $F$.
So $(X, \frac{R_{f|_F}}{q-1})$ is lc.
\end{proof}

We give some partial results on the $\mathbb{P}^2$ case which is enough for the proof of Theorem \ref{main-thm-lcy}.
\begin{theorem}\label{thm-p2-ticurve}
Let $f:X\cong \mathbb{P}^2\to X$ be a $q$-polarized endomorphism.
Suppose $T_f\neq\emptyset$.
Then $(X, \frac{R_f}{q-1})$ is log Calabi-Yau after iteration.
\end{theorem}
\begin{proof}
Write $\frac{R_f}{q-1}=T_f+\Delta_f$.
Note that $\Delta_f$ and $T_f$ have no common components.
By Theorem \ref{thm-coefficient}, we may assume that $\Delta_f$ has coefficient $<1$ after sufficient iteration of $f$.
Then the non-klt locus $\text{Nklt}(X, \frac{R_f}{q-1})=T_f\cup S$ where $S$ is a finite set of points outside $T_f$.

We first show that the non-lc locus $\text{Nlc}(X, \frac{R_f}{q-1})\cap T_f=\emptyset$.
In particular, if $S=\emptyset$, then $(X, \frac{R_f}{q-1})$ is lc.
Suppose the contrary and let $x\in \text{Nlc}(X, \frac{R_f}{q-1})\cap T_f$.
Let $C$ be an irreducible component of $T_f$ containing $x$.
By \cite{Gur03}, $C$ is a line.
After iteration, we may assume $f^{-1}(C)=C$.
Consider the log ramification divisor formula,
$$(K_X+C)=f^*(K_X+C)+R_f-(q-1)C.$$
Apply adjunction on $C$, we have
$$K_C=(f|_C)^*K_C+(R_f-(q-1)C)|_C=(f|_C)^*K_C+R_{f|_C}.$$
So $(R_f-(q-1)C)|_C=R_{f|_C}$.
Note that $\frac{R_{f|_C}}{q-1}$ has coefficient $\le 1$.
Then $((\frac{R_f}{q-1}-C)\cdot C)_x\le 1$.
By inverse of adjunction (cf.~\cite[Corrolary 5.57]{KM98}), $(X,\frac{R_f}{q-1})$ is lc near $x$, a contradiction.

Assume now that $S\neq \emptyset$.
Then $\text{Nklt}(X, \frac{R_f}{q-1})$ is not connected.
By \cite[Theorem 1.2]{HH19}, $(X, \frac{R_f}{q-1})$ is plt, a contradiction.
\end{proof}

\begin{corollary}\label{cor-p2-ticurve}
Let $f:X\cong \mathbb{P}^2\to X$ be a $q$-polarized endomorphism.
Suppose there is a birational morphism $\pi:W\to X$ from a normal projective surface $W$ such that $\pi\circ h=f\circ \pi$ for some surjective endomorphism $h:W\to W$.
Suppose further either the exceptional locus of $W$ is reducible or $W$ has more than one negative curves.
Then $T_f\neq\emptyset$ and hence $(X, \frac{R_f}{q-1})$ is log Calabi-Yau after iteration.
\end{corollary}
\begin{proof}
By Proposition \ref{prop-bir-fanotype}, $W$ is of Fano type and hence $\Q$-factorial (cf.~\cite[Proposition 4.11]{KM98}).
By \cite[Lemma 4.3]{MZ22}, the negative curves on $W$ are $h^{-1}$-periodic.
Note that the existence of negative curve on $W$ not contracted by $\pi$ will cause $T_f\neq \emptyset$.
So we may always assume that the exceptional locus of $W$ is reducible.

Let $E=\sum\limits_{i=1}^n E_i$ be the reduced $\pi$-exceptional divisor with $n>1$.
Write $K_W=\pi^*K_X+\sum\limits_{i=1}^n a_i E_i$ where $a_i>0$ since $X$ is smooth. 
By the negativity lemma, $K_W$ is not $\pi$-nef.
So we can run $h$-equivariant (after iteration) relative minimal model program of $W$ over $X$ which finally ends up with $X$.
Denote by $\tau:W_1\to X$ the last step and $\sigma:W\to W_1$ the composition of the previous steps.
Note that $W_1$ is of Fano type and $\rho(W_1)=2$.
Denote by $p:W_1\to Y$ be the contraction induced by the extremal ray different with $\tau$.
If $p$ is birational, then the exceptional curve of $p$ is not contracted by $\tau$ and hence $T_f\neq\emptyset$.
So we may assume that $p:W_1\to Y\cong \mathbb{P}^1$ is a $\mathbb{P}^1$-fibration.
Suppose that $E_1$ is contracted by $\sigma$.
Then $p(\sigma(E_1))$ is an $(h|_Y)^{-1}$-invariant point on $Y$ (cf.~\cite[Lemma 7.5]{CMZ20}).
Therefore, $p^{-1}(p(\sigma(E_1)))$ is an $(h|_{W_1})^{-1}$-invariant curve which is not contracted by $\tau$ and hence $T_f\neq\emptyset$.
Finally, we apply Theorem \ref{thm-p2-ticurve}.
\end{proof}

\section{Anti-canonical divisor}\label{sec-big}
Throughout this section, we use the following setting:

\begin{blank}\label{set-rel2}
Let $f:X\to X$ be a $q$-polarized endomorphism of a smooth projective threefold $X$ admitting an ($f$-equivariant) Fano contraction $\pi:X\to Y$ with $Y$ being an elliptic curve and general fibres being $\mathbb{P}^2$.
In this setting, $\rho(X)=2$ and we may further assume $f^*|_{\N^1(X)}=q \id$.
Denote by
$$\phi:X\dashrightarrow Z$$ 
the Chow reduction of the Iitaka fibration of $-K_X$ which is $f$-equivariant by \cite[Theorem 7.8]{MZ22}.
\end{blank}

\begin{proposition}\label{prop-semiample}
Suppose there is another fibration $\tau:X\to V$ different with $\pi$ with $0<\dim(V)<3$.
Then $-K_X$ is semi-ample.
In particular, $X$ is of Calabi-Yau type.
\end{proposition}
\begin{proof}
By the assumption of the setting \ref{set-rel2}, $\rho(V)=1$ because $\rho(X)=2$ and $-K_X$ is not ample because otherwise $X$ is Fano and hence rationally connected contradicting with the setting of a Fano fibration over an elliptic curve.
Since $\pi$ and $\tau$ are two different fibrations to positive lower dimensional varieties, the two extremal rays of $\Nef(X)$ are generated by the pullbacks of ample divisors on $Y$ and $V$ which are not big.
Hence, $\Nef(X)=\PE^1(X)$.
Note that $-K_X$ is not ample but effective and $\pi$-ample.
So $-K_X\equiv \phi^*H$ for some ample $\Q$-divisor $H$ on $V$.
By Bertini's theorem, we may assume $(X, \phi^*H)$ is lc (even terminal) after a suitable choice of $H$.
By the abundance, $-K_X\sim_{\Q} \phi^*H$ and we are done.
\end{proof}

\begin{theorem}\label{thm-kappa1}
Assume further $\kappa(X, -K_X)<3$, i.e., $-K_X$ is not big.
Then $X$ is of Calabi-Yau type.
\end{theorem}
\begin{proof}
By Proposition \ref{prop--kkappa0}, we may assume $\kappa(X, -K_X)>0$.
By Proposition \ref{prop-semiample}, we may assume $\phi$ is not well-defined.
Let $W$ be the normalization of the graph of $\phi$. 
Denote by $h:W\to W$ the lifting of $f$ and $\sigma:W\to X$ the induced birational morphism.
Since $X$ is smooth, the exceptional locus of $\sigma$ is of pure codimension one in $W$ (cf.~\cite[Corollary 2.63]{KM98}).
We denote it by $E:=\sum\limits_{i=1}^n E_i$ with $n>0$.
After iteration, we may assume $E_i$ is $h^{-1}$-invariant.
Note that the elliptic curve $Y$ admits no $g^{-1}$-periodic points.
So $\pi(\sigma(E_i))=Y$.

Let $W_y:=(\pi\circ \sigma)^{-1}(y)$ be an $h$-invariant general fibre after iteration (cf. \cite[Theorem 5.1]{Fak03}).
Denote by $X_y:=\pi^{-1}(y)\cong \mathbb{P}^2$.
If $E$ is reducible, i.e., $n>1$, then $E\cap W_y$ is also reducible.
By Corollary \ref{cor-p2-ticurve}, $(X_y, \frac{R_{f|_{X_y}}}{q-1})$ is log Calabi-Yau after iteration of $f$.
By Proposition \ref{prop-cbf-rf}, $(X, \frac{R_f}{q-1})$ is log Calabi-Yau after iteration of $f$.
So we may assume $E$ is irreducible.
Note that $E$ is $\Q$-Cartier (cf.~\cite[Lemma 2.62]{KM98}).
Then $W$ is $\Q$-factorial.
By \cite[Theorem 1.3]{Zha13}, $(W,E)$ is lc.
Since $W\backslash E$ is smooth, $W$ is klt.

Write $K_W=\sigma^*K_X+aE$ with $a>0$ since $X$ is smooth.
Then $$0\le \kappa(W, -K_W)\le \kappa(W, -\sigma^*K_X)=\kappa(X, -K_X)$$
where the first inequality follows from Theorem \ref{thm-move} and the last equality follows from \cite[Theorem 5.13]{Uen75}.
Let $\phi_W:W\dashrightarrow V$ be the Chow reduction of the Iitaka fibration of $-K_W$.
So $\dim(V)<3$.
Suppose $\phi_W$ is not well-defined.
Let $\widetilde{W}$ be the normalization of the graph of $\phi_W$ and $\widetilde{\sigma}:\widetilde{W}\to W$ the induced birational morphism.
Since $W$ is $\Q$-factorial, the exceptional locus of $\widetilde{\sigma}$ is of pure codimension one in $\widetilde{W}$ (cf.~\cite[Corollary 2.63]{KM98}) and each irreducible component dominates $Y$ by a similar argument.
In particular, the restricted birational morphism $\widetilde{W}_y\to X_y$ has reducible exceptional locus.
Then we are done by Corollary \ref{cor-p2-ticurve} and Proposition \ref{prop-cbf-rf}.
So we may assume that $\phi_W$ is well-defined, $\rho(W_y)=2$ and $W_y$ has only one negative curve $C=E\cap W_y$, i.e., the exceptional curve of $W_y\to X_y$. 
We give more infomation on $W_y$.
By Proposition \ref{prop-bir-fanotype} and \cite[Proposition 4.11]{KM98}, $W_y$ is of Fano type and $\Q$-factorial.
Since $W_y$ has only one negative curve (which is also $K_{W_y}$-negative since $X$ is smooth), another contraction of $W_y$ is a $\mathbb{P}^1$-fibration over $\mathbb{P}^1$.
In particular, $-K_{W_y}$ is ample.

If $\dim(V)=0$, then $(W,T_h)$ is log Calabi-Yau by Proposition \ref{prop--kkappa0}.
If $\dim(V)=1$, then $\phi_W$ is just the Iitaka fibration of $-K_W$ and hence $-K_W$ is semi-ample.
In both cases, $W$ and hence $X$ are of Calabi-Yau type by Lemma \ref{lem-cy-bir}.

We are left with $\dim(V)=2$. 
Note that $-K_W|_{W_y}=-K_{W_y}$ is ample.
So the restriction $\phi_W|_{W_y}:W_y\to V$ is surjective.
We first assume that the induced map $X_y\dashrightarrow V$ is not well-defined.
By the rigidity lemma (cf.~\cite[Lemma 1.15]{Deb01}), $\phi_W|_{W_y}$ contracts no curves. 
Then $\phi_W|_{W_y}:W_y\to V$ is finite surjective.
By \cite[Lemma 5.16 and Proposition 5.20]{KM98}, $V$ is klt and $\Q$-factorial.
Let $C_V:=\phi_W|_{W_y}(C)$.
By the projection formula, $C_V$ is also a negative curve.
Then, we have an $h|_V$-equivariant divisorial contraction $V\to V_1$ which contracts $C_V$.
By the rigidity lemma again, the induced rational map $X_y\dashrightarrow V_1$ is then well-defined.
If the map $X_y\dashrightarrow V$ is already well-defined, then we simply identify $V_1$ with $V$.
Consider the induced $f$-equivariant dominant rational map $\phi':X\dashrightarrow V_1$.
Note that indeterminant locus of $\phi'$ is $f^{-1}$-invariant which does not dominate $Y$ since $\phi'$ is well-defined near $X_y$.
By \cite[Lemma 7.5]{CMZ20}, $\phi'$ is well-defined and we are done by Proposition \ref{prop-semiample}.
\end{proof}

\section{Proof of Theorem \ref{main-thm-lcy}}

We refer to \cite{Ati57} for well-known facts on vector bundles over elliptic curves.
By $\mathcal{F}_n$ we mean the unique indecomposable rank $n$ vector bundle with non-trivial global sections over an elliptic curve.
The following lemmas will be used later.

\begin{lemma}\label{lem-f2}
Let $S\cong \mathbb{P}_Y(\mathcal{F}_2)$ where $Y$ is an elliptic curve and $\mathcal{F}_2$ is the unique indecomposable rank 2 vector bundle with non-trivial global sections.
Then $S$ admits no polarized endomorphism.
\end{lemma}
\begin{proof}
Suppose the contrary that there is a $q$-polarized endomorphism $f:S\to S$.
We may assume $f^*|_{\N^1(S)}=q\id$.
Note that $\pi:S\to Y$ has a unique section $C$ with $C^2=0$ and $-K_S\sim 2C$.
Then $\mathcal{O}_S(1)\cong \mathcal{O}_S(C)$ and $\pi_*\mathcal{O}_S(n)\cong \Sym^n(\mathcal{F}_2)\cong \mathcal{F}_{n+1}$ for $n>0$ (cf.~\cite[Theorem 9]{Ati57}).
In particular, $h^0(S,nC)=h^0(Y,\mathcal{F}_{n+1})=1$ and $\kappa(S,-K_S)=0$.
By Proposition \ref{prop--kkappa0}, we have $\Supp R_f=T_f$ and $-K_S\sim_{\Q} T_f$ .
Note that $T_f$ is reduced.
So $T_f\neq 2C$ and hence $\kappa(S, C)>0$, a contradiction.
\end{proof}

\begin{proof}[Proof of Theorem \ref{main-thm-lcy}]
{\bf Step 1.} 
In this step, we reduce our situation to the case when $T_f=\emptyset$, $f^*|_{\N^1(X)}=q\id$ and $X\cong \mathbb{P}_Y(\mathcal{E})$ for some rank 3 vector bundle $\mathcal{E}$ on an elliptic curve $Y$.

By \cite[Theorem 1.8]{MZ18}, we run $f$-equivariant minimal model program (after iteration)
$$X=X_1\dashrightarrow X_2\dashrightarrow\cdots\dashrightarrow X_r\to Y$$
where $Y$ is a $Q$-abelian variety and $g:=f|_Y$ which is still polarized.
Moreover, each composition $X_i\dashrightarrow Y$ is a well-defined equi-dimensional morphism with all the fibres irreducible and rationally connected.
Denote by $\pi:X\to Y$ be the induced composition.
If $\dim(Y)=3$, then $X=Y$ is $Q$-abelian and hence $(X,\frac{R_f}{q-1})=(X,0)$ is log Calabi-Yau.
If $\dim(Y)=0$, then $X$ is of Fano type by Theorem \ref{thm-yos}.

Suppose $\dim(Y)=2$.
By Theorem \ref{thm-move}, we have $-(K_X+T_f)\sim_{\Q} P$ for some effective $\Q$-divisor $P$ with coefficients $\le 1$ and containing no component of $T_f$.
Let $F$ be a general fibre of $\pi$.
Since $F$ is general, we may assume that $\Supp P \cup T_f$ intersects with $F$ transversally.
In particular, $(F, (T_f+P)|_F)$ is lc and $(X,T_f+P)$ is lc over the generic point of $\pi$.
By the canonical bundle formula \cite[Theorem 4.1.1]{Fuj04},
we have 
$$0\equiv K_X+T_f+P=\pi^*(K_Y+B+M)$$
Note that $B$ is effective and $M$ is pseudo-effective and $K_Y\equiv 0$.
So $B=0$ and hence $(X, T_f+P)$ is lc and $X$ is of Calabi-Yau type.

Therefore, we may assume that $\dim(Y)=1$.
Then $f^*|_{\N^1(X)}=q\id$ after iteration by \cite[Theorem 1.8]{MZ18}.
Note that the general fibre $F$ of $\pi$ is a smooth rational surface.
If $\rho(F)>1$, then we may apply  Theorems \ref{thm-surf}  to  Proposition \ref{prop-cbf-rf} and hence $X$ is log Calabi-Yau.
So we may assume $\rho(F)=1$ and hence $F\cong \mathbb{P}^2$.
Since $X$ is smooth, the first step $X_1\dashrightarrow X_2$ is either a divisorial contraction of a Fano contraction by \cite[Theorem 3.3]{Mor82}.
If it is divisorial, then the exceptional divisor $E$ is disjoint with the general fibre $F$ since now $F\cong \mathbb{P}^2$ and hence $\pi(E)$ is a point in $Y$.
However, after iteration, $f^{-1}(E)=E$ and hence $g^{-1}(\pi(E))=\pi(E)$ by \cite[Lemma 7.5]{CMZ20}, a contradiction by \cite[Lemma 4.7]{MZ18}.
So $X_1\dashrightarrow X_2$ is a Fano contraction over $Y$.
Again because $F\cong \mathbb{P}^2$, this is possible only when $X_2=Y$.
In particular, $\rho(X)=2$ and $\pi$ is a Fano contraction.
By \cite[Theorem 3.5]{Mor82} and since the Brauer group of the elliptic curve $Y$ is trivial, $X\cong \mathbb{P}_Y(\mathcal{E})$ for some rank 3 vector bundle $\mathcal{E}$ on $Y$.

Suppose $f^{-1}(P)=P$ for some prime divisor $P$. Then $P$ dominates $Y$ since $f|_Y$ is \'etale and polarized.
Let $F$ be an $f$-invariant (after iteration) fibre of $\pi$.
Then $D\cap F\subseteq T_{f|_F}$.
By Theorem \ref{thm-p2-ticurve} and Proposition \ref{prop-cbf-rf}, $(X, \frac{R_f}{q-1})$ is log  Calabi-Yau.
So we may assume $T_f=\emptyset$.

{\bf Step 2.} 
In this step, we reduce our situation to the case when $\mathcal{E}\cong \mathcal{F}_2\oplus \mathcal{L}$ for some line bundle $\mathcal{L}$.

Note that $f|_Y$ is polarized and has Zariski dense periodic points.
So after iteration and choosing some fixed point as the identity element of $Y$, we may assume $f|_Y$ is an isogeny.
In particular, $f|_Y$ commutes with any multiplication map.
By Lemma \ref{lem-cy-qe}, we can always replace $X$ by base change of multiplication map (by 3) on $Y$. 
So we may assume $3| \deg \mathcal{E}$.
By Theorem \ref{thm-kappa1}, we may further assume $-K_X$ is big.
Consider the following two cases.

Suppose $\mathcal{E}\cong \mathcal{L}_1\oplus \mathcal{L}_2\oplus \mathcal{L}_3$ where $\mathcal{L}_1\cong \mathcal{O}_Y$.
Let $D_1$ be the divisor determined by the projection $\mathcal{E}\to \mathcal{L}_2\oplus \mathcal{L}_3$ and define $D_2$ and $D_3$ similarly.
Consider the exact sequnce
$$0\to \mathcal{O}_X(1)\otimes\mathcal{O}_X(-D_1)\to \mathcal{O}_X(1)\to \mathcal{O}_X(1)\otimes \mathcal{O}_{D_1}\to 0.$$
Taking $\pi_*$, we have
$$0\to \pi_*(\mathcal{O}_X(1)\otimes\mathcal{O}_X(-D_1))\to \mathcal{E}\to \pi_*\mathcal{O}_{D_1}(1)=\mathcal{L}_2\oplus \mathcal{L}_3\to 0$$
with $0$ on the right because $R^1\pi_*(\mathcal{O}_X(1)\otimes\mathcal{O}_X(-D_1))=0$ since all the fibres are $\mathbb{P}^2$.
By the projection formula, we have $\mathcal{O}_X(1)\cong \mathcal{O}_X(D_1)$ and similarly, $D_1\sim D_2+\pi^*c_1(\mathcal{L}_2)$ and $D_1\sim D_3+\pi^*c_1(\mathcal{L}_3)$.
So by the relative Euler sequence, we have
$$-K_X\sim 3D_1-\pi^*(c_1(\mathcal{E}))=D_1+D_2+D_3.$$
It is easy to see that $D_1+D_2+D_3$ has simple normal crossing.
In particular, $(X,D_1+D_2+D_3)$ is a log Calabi-Yau pair.

Suppose $\mathcal{E}$ is indecomposable. Since $3| \deg \mathcal{E}$, we have $\mathcal{E}\cong \mathcal{F}_3\otimes \mathcal{L}=0$ for some line bundle $\mathcal{L}$ (cf.~\cite[Theorem 5]{Ati57}).
Therefore, we may assume $\mathcal{E}\cong \mathcal{F}_3$.
However, $-K_X$ is then nef but not big, a contradiction with our assumption.

Now we may assume $\mathcal{E}\cong \mathcal{F}\oplus \mathcal{L}$ where $\mathcal{F}$ is an indecomposable rank 2 vector bundle and $\mathcal{L}$ is a line bundle.
Note that we may also assume that $\mathcal{F}$ does not split after taking pullback of the multiplication map (by 2) of $Y$.
So we may assume that $\mathcal{F}\cong \mathcal{F}_2$.

{\bf Step 3.} 
In this step, we show that $\deg \mathcal{L}<0$.
Note that there is a non-splitting exact sequence
$$0\to \mathcal{O}_Y\to \mathcal{E}\to  \mathcal{O}_Y\oplus \mathcal{L}\to 0$$
which is induced by the non-splitting exact sequence of $\mathcal{F}_2$.
Let $D$ be the divisor determined by $\mathcal{E}\to  \mathcal{O}_Y\oplus \mathcal{L}$.
Then $\mathcal{O}_X(1)\cong \mathcal{O}_X(D)$.
Let $D'$ be the divisor determined by the natrual projection $\mathcal{E}\to \mathcal{F}_2$.
Then $D\sim D'+\pi^*c_1(\mathcal{L})$ and we have
$$-K_X\sim 3D-\pi^*c_1(\mathcal{L})\sim 2D+D'.$$
Note that $D|_D\sim C$ where $C$ is the section curve determined by $\mathcal{O}_Y\oplus \mathcal{L}\to \mathcal{L}$.
Then we see that $-K_X$ is big if and only if $\deg \mathcal{L}\neq 0$.
By Lemma \ref{lem-f2}, $D'$ is not $f$-periodic.
Note that $f^*D'\equiv qD'$ and $f^{-1}(D')\neq D'$.
So $D'|_{D'}$ is pseudo-effective and hence nef on $D'$.
Then $D'$ is nef.
Note that $-K_X$ is not nef because otherwise $X$ is of Fano type.
So $D$ is not nef and hence $\deg \mathcal{L}<0$.

{\bf Step 4.} 
In this step, we show that $\kappa(X,D)=0$.
Note that 
$$\pi_*(\mathcal{O}_X(n))\cong \Sym^n(\mathcal{E})\cong \bigoplus_{i=0}^n \mathcal{F}_{i+1}\otimes \mathcal{L}^{\otimes (n-i)}$$
for $n\ge 0$.
Then 
$$h^0(X, nD)=\sum_{i=0}^n h^0(Y,\Sym^i(\mathcal{F}_2)\otimes \mathcal{L}^{\otimes (n-i)})=h^0(Y,\mathcal{F}_{n+1})=1$$
by noticing that $h^0(Y,\mathcal{F}_{i+1}\otimes \mathcal{L}^{\otimes (n-i)})=0$ for any $i<n$.
So $\kappa(X,D)=0$.

{\bf End of the proof.}
By Proposition \ref{prop-kappa0}, $D$ is $f^{-1}$-periodic.
However, this contradicts the reduction we obtained in Step 1.
\end{proof}

\begin{remark}
When $X\cong \mathbb{P}_Y(\mathcal{F}_2\oplus\mathcal{L})$ for some line bundle $\mathcal{L}$ of negative degree,
we do not know whether $X$ is of Calabi-Yau type or not. 
What we proved above is that either $X$ is of Calabi-Yau type or $X$ admits no polarized endomorphism.
Therefore, it will be very interesting to study the following question. 

\end{remark}

\begin{question}
Let $X\cong \mathbb{P}_{Y}(\mathcal{E})$ be a projective bundle over an elliptic curve $Y$. 
\begin{enumerate}
\item When will $X$ be of Calabi-Yau type?
\item When will $X$ admit a polarized endomorphism?
\end{enumerate}
\end{question}


\begin{thebibliography}{99}
%
%
%
%

\bibitem[Ati57]{Ati57}
M.~F.~Atiyah,
Vector bundles over an elliptic curve, 
Proc. London Math. Soc. (3) \textbf{7} (1957), 414--452. 


\bibitem[BG17]{BG17}
A.~Broustet and Y.~Gongyo, 
Remarks on log Calabi-Yau structure of varieties admitting polarized endomorphisms,
Taiwanese J. Math. \textbf{21} (2017), no. 3, 569--582.


\bibitem[CMZ20]{CMZ20}
P.~Cascini, S.~Meng and D.-Q.~Zhang,
Polarized endomorphisms of normal projective threefolds in arbitrary characteristic,
Math. Ann. \textbf{378} (2020), no. 1--2, 637--665.


\bibitem[Deb01]{Deb01}
O.~Debarre,
Higher-dimensional algebraic geometry,
Universitext, Springer-Verlag,
New York,
2001.


\bibitem[Fak03]{Fak03}
N.~Fakhruddin,
Questions on self-maps of algebraic varieties,
J. Ramanujan Math. Soc., \textbf{18}(2):109--122, 2003.

\bibitem[Fuj04]{Fuj04}
O. Fujino, Higher direct images of log canonical divisors, arXiv:math/0302073, J. Differential Geom. 66 (2004), no. 3, 453–479.

\bibitem[Gon13]{Gon13}
Y.~Gongyo,
Abundance theorem for numerically trivial log canonical divisors of semi-log canonical pairs,
J. Algebraic Geom. \textbf{22} (2013), no. 3, 549--564. 

\bibitem[GKP16]{GKP16}
D. Greb, S. Kebekus and T. Peternell, \'Etale fundamental groups of Kawamata log terminal spaces, flat sheaves, and quotients of abelian varieties, Duke Math. J. \textbf{165}(2016), no. 10, 1965--2004.

\bibitem[Gur03]{Gur03}
R. V. Gurjar, On ramification of self-maps of {$\bold P^2$}, J. Algebra, \textbf{259} (2003), no. 1, 191--200.

\bibitem[HH19]{HH19}
C.~Hacon and J. Han,
On a connectedness principle of Shokurov-Koll\'ar type,
Sci. China Math. \textbf{62} (2019), no. 3, 411--416.



\bibitem[KM98]{KM98}
J.~Koll\'ar and S.~Mori,
Birational geometry of algebraic varieties,
Cambridge Univ. Press, 1998.


\bibitem[Men20]{Men20}
S.~Meng,
Building blocks of amplified endomorphisms of normal projective varieties,
Math.~Z. \textbf{294} (2020), no. 3, 1727--1747.


\bibitem[MZ18]{MZ18}
S.~Meng and D.-Q.~Zhang,
Building blocks of polarized endomorphisms of normal projective varieties,
Adv.~Math. \textbf{325} (2018), 243--273.


\bibitem[MZ20]{MZ20}
S.~Meng and D.-Q.~Zhang,
Semi-group structure of all endomorphisms of a projective variety admitting a polarized endomorphism,
Math.~Res.~Lett. \textbf{27} (2020), no. 2, 523--550.

\bibitem[MZ22]{MZ22}
S.~Meng and D. -Q.~Zhang,
Kawaguchi-Silverman conjecture for surjective endomorphisms,
Doc. Math. \textbf{27}(2022), 1605-1642.


\bibitem[Mor82]{Mor82}
S. Mori, Threefolds whose canonical bundles are not numerically effective, Ann. of Math. (2) \textbf{116} (1982), no. 1, 133--176.

\bibitem[Nak02]{Nak02}
N.~Nakayama,
Ruled surfaces with non-trivial surjective endomorphisms,
Kyushu J. Math., \textbf{56} (2002), 433--446.

\bibitem[Nak04]{Nak04}
N.~Nakayama, 
Zariski-decomposition and abundance, MSJ Memoirs Vol. \textbf{14}, Math. Soc. Japan,
2004.


\bibitem[Uen75]{Uen75}
K.~Ueno,
Classification theory of algebraic varieties and compact complex spaces, Lecture Notes in Mathematics, Vol. \textbf{439}, Springer-Verlag, Berlin, 1975, Notes written in collaboration with P. Cherenack.


\bibitem[Yos20]{Yos20}
S. Yoshikawa, Structure of fano fibrations of varieties of varieties admitting an int-amplified endomorphism,  Adv. Math. \textbf{391} (2021), Paper No. 107964, 32 pp.

\bibitem[Zha13]{Zha13}
D.-Q. Zhang, Invariant hypersurfaces of endomorphisms of projective varieties, Adv. Math., 2013, \textbf{252}(3):185--203.

\end{thebibliography}
\end{document}